\documentclass{compositio}

\usepackage[utf8]{inputenc}
\usepackage[english]{babel}
\usepackage{newlfont}
\usepackage{amsmath,amssymb,amsfonts,amsbsy,bbm}
\usepackage{mathtools,braket,mathrsfs}
\usepackage[all]{xy}
\usepackage{tikz}
\usetikzlibrary{arrows}

\usepackage{stmaryrd}
\usepackage{upgreek}
\usepackage[new]{old-arrows}
\usepackage{extarrows}
\usepackage{textcomp}

\newcommand{\plectic}[0]{\text{\textmarried}}

\newcommand{\bb}{\mathbb}

\newcommand{\scr}{\mathscr}
\newcommand{\mrm}{\mathrm}
\newcommand*{\bfcdot}{\scalebox{0.6}{$\bullet$}}

\newcommand{\Z}{\ensuremath{\mathbb{Z}}}

\newcommand{\Q}{\ensuremath{\mathbb{Q}}}
\newcommand{\R}{\ensuremath{\mathbb{R}}}
\newcommand{\C}{\ensuremath{\mathbb{C}}}

\newcommand{\too}{\longrightarrow}								

\newcommand{\n}{\ensuremath{\mathfrak{n}}}

\renewcommand{\det}{\operatorname{det}}

\def\XXint#1#2#3{{\setbox0=\hbox{$#1{#2#3}{\int}$}%
\vcenter{\hbox{$#2#3$}}\kern-.5\wd0}}%

\theoremstyle{plain}
\newtheorem{theorem}{Theorem}[section]
\newtheorem{lemma}[theorem]{Lemma}
\newtheorem{proposition}[theorem]{Proposition}
\newtheorem{corollary}[theorem]{Corollary}
\newtheorem{conjecture}[theorem]{Conjecture}

\newtheorem{thmx}{Theorem}

\theoremstyle{definition}
\newtheorem{remark}[theorem]{Remark}

\newtheorem{definition}[theorem]{Definition}

\def\XXint#1#2#3{{\setbox0=\hbox{$#1{#2#3}{\int}$ }
\vcenter{\hbox{$#2#3$ }}\kern-.585\wd0}}

\begin{document}

\title{Plectic Jacobians}

\author{Michele Fornea}
\email{mfornea.research@gmail.com}
\address{MSRI/SLMath, Berkeley, CA, USA.}


\begin{abstract}
Looking for a geometric framework to study plectic Heegner points, we define a collection of abelian varieties -- called \emph{plectic Jacobians} -- using the middle degree cohomology of quaternionic Shimura varieties (QSVs). The construction is inspired by the definition of Griffiths' intermediate Jacobians and rests on Nekov\'a$\check{\text{r}}$--Scholl's notion of plectic Hodge structures. Moreover, we construct exotic Abel--Jacobi maps sending certain zero-cycles on QSVs to plectic Jacobians. 
\end{abstract}

\maketitle

\tableofcontents

\section{Introduction}
For a long time number theorists have been looking for suitable generalizations of Heegner points to tackle the BSD--conjecture for elliptic curves of rank greater than $1$. Motivated by that problem, a conjectural construction of determinants of global points was recently proposed (\cite{PlecticInvariants}, \cite{plecticHeegner}) combining Darmon's pioneering work (\cite{IntegrationDarmon}) with the powerful insights of Nekov\'a$\check{\text{r}}$--Scholl's plectic conjectures (\cite{PlecticNS}). These \emph{plectic Stark--Heegner (PSH) points} are constructed using $p$-adic integration and their peculiar appearance is motivated by the uniformization of QSVs by certain $p$-adic symmetric domains.  
To explain how PSH points should arise  from global points, precise conjectures were formulated (\cite{plecticHeegner}, Conjectures 1.3 \& 1.5) which, in a nutshell, claim that given an elliptic curve of algebraic rank $r$, a PSH point constructed using $r$ different $p$-adic places is in the image of the top exterior power of the Mordell--Weil group via a $p$-adic determinant map. Notably, for elliptic curves of rank $1$ those conjectures recover the expectation that classical Stark--Heegner points are images  of  global points under $p$-adic localization. 

\noindent Those conjectures on PSH points were substantiated by numerical and theoretical evidence. On the computational side, in \cite{PlecticInvariants} they were verified (up to precision) for several elliptic curves of rank $2$ defined over $\bb{Q}(\sqrt{13})$ and $\bb{Q}(\sqrt{37})$.  
On the theoretical side, instances of the conjectures were proved in the setting of polyquadratic CM extensions  (\cite{polyquadraticPlectic}) leveraging higher $p$-adic Gross--Zagier formulas for anticyclotomic $p$-adic $L$-functions (\cite{plecticHeegner}, Theorem A). Moreover, it is reasonable to expect that the recent work of Molina--Hernandez (\cite{hernMOLINA}) will help in clarifying the connection between PSH points and  generalized Kato classes (\cite{DR2}).

\noindent As is the case for Darmon's Stark--Heegner points, one cannot usually guarantee that PSH points arise from global points because of their inherently analytic construction. More than 20 years after the introduction of  Stark--Heegner points, our understanding of their conjectural global properties remains quite unsatisfactory in general. There is a notable exception: for CM extensions Darmon's points recover classical Heegner points, whose global features have long been understood using Jacobian varieties and the theory of complex multiplication. 
One of the appealing traits of PSH points is that they are already interesting and new for CM extensions. Thus, from now on, we will refer to PSH points for CM extensions as \emph{plectic Heegner points}, and we will try to shed some light on their attributes using Nekov\'a$\check{\text{r}}$--Scholl's \emph{plectic Hodge theory}.

\subsection{Main results}
 
 Nekov\'a$\check{\text{r}}$ and Scholl observed (\cite{PlecticNS}, \cite{PlecticMHS}) that Hodge structures of Hilbert modular varieties carry more information than those of general K\"ahler manifolds. In particular, they showed the existence of a K\"unneth-like structure that reflects the canonical demposition of the tangent bundle of Hilbert modular varieties. 
 \begin{definition}
     Let $n\ge1$ be an integer. An $n$-plectic Hodge structure on a finite free $\Z$-module $H$ consists in a decomposition
\[
H\otimes_\Z\C=\bigoplus_{\alpha,\beta\in\bb{Z}^n}H^{\alpha,\beta}\quad\text{such that}\quad H^{\alpha,\beta}=\overline{H^{\beta,\alpha}}.
\]  
 \end{definition}
 \begin{remark}
     There is a natural procedure that produces a Hodge structure from the data of an $n$-plectic Hodge structure. Given $\alpha=(\alpha_j)_{j=1}^n\in\Z^n$ set $\lvert\alpha\rvert=\sum_{j=1}^n\alpha_j$, and let $H$ be an $n$-plectic Hodge structure. The classical Hodge structure arising from $H$ is defined by setting 
     \[
H^{p,q}:=\bigoplus_{\lvert\alpha\rvert=p, \lvert\beta\rvert=q} H^{\alpha,\beta}\qquad \forall\ p,q\in\bb{Z}.
     \]
     In this case we say that the $n$-plectic Hodge structure refines the associated Hodge structure.
 \end{remark}

\noindent Our first main theorem shows that plectic Hodge structures arise in the cohomology of compact rigidified K\"ahler manifolds, i.e., compact complex manifolds endowed with an integrable foliation and a compatible K\"ahler metric (see Definitions \ref{rigidCM}, \ref{def RIG} and Corollary \ref{plectic Hodge structure}). 
\begin{thmx}\label{THM A}
	Let $X$ be an $n$-dimensional compact rigidified K\"ahler manifold. The cohomology of $X$ is endowed with a canonical $n$-plectic Hodge structure refining its classical Hodge structure.
\end{thmx}
\noindent We note here that the compatibility conditions between the given integrable foliation and K\"ahler metric are singled out to ensure that the Laplacian operator associated to the K\"ahler metric respects the decomposition of harmonic differential forms induced by the foliation.

 \subsubsection{Plectic Jacobians and exotic Abel--Jacobi maps.}
 Our work on PSH points was inspired by Nekov\'a$\check{\text{r}}$ and Scholl's belief that CM points on higher dimensional QSVs could be used to study the arithmetic of elliptic curves of higher rank. 
While previous articles leveraged $p$-adic techniques, this paper begins to develop an Archimedean framework to study plectic Heegner points following Oda's trailblazing work on periods of Hilbert modular varieties (\cite{Oda}, \cite{OdaGeneral}). The aim is to
 understand a form of geometric modularity where Jacobians of Shimura curves are replaced by \emph{plectic Jacobians} of higher dimensional QSVs. As the Jacobian of a curve $C$ can be constructed from the weight $1$ Hodge structure $\mrm{H}^1(C,\Z)$, so plectic Jacobians of a QSV $X$ are defined using the plectic Hodge structure appearing in the middle degree cohomology group $\mrm{H}^{\mrm{dim}X}(X,\Z)$.
 \begin{definition}
An $n$-plectic Hodge structure $H$ is \emph{effective} and has \emph{weight} $\underline{1}=(1,\hspace{-0.4mm}..,1)\in\bb{Z}^n$ if 
\[
H^{\alpha,\beta}\not=0\quad\implies\quad \alpha,\beta\in\bb{N}^n\quad\&\quad\alpha+\beta=\underline{1}.
\]
\end{definition}

An effective $n$-plectic Hodge structure of weight  $\underline{1}\in\bb{Z}^n$ can be thought of as a collection of $n$ effective Hodge structures of weight $1$ on the same underlying module by setting 
\[
F^{1_j}=F^{1_j}(H):=\bigoplus_{\alpha_j\ge 1}H^{\alpha,\beta}\qquad \text{for any } j=1,\dots, n.
\]
It is then natural to make the following definition. 
\begin{definition}
	Let $H$ be an effective $n$-plectic Hodge structure of weight $\underline{1}\in\bb{Z}^{n}$. For any $j=1,\dots,n$ the plectic Jacobian $\mrm{J}_{\plectic}(H,j)$ associated to $H$ is the complex torus defined by 
	\[
	\mrm{J}_{\plectic}(H,j):=H\backslash (H\otimes_\Z\C)/F^{1_j}.
	\]
\end{definition}
Systems of Hecke eigenvalues of modular elliptic curves appear in the cohomology of QSVs only in middle degree, and the cuspidal part of those middle degree cohomology groups can be shown to carry a canonical effective plectic Hodge structure of weight $\underline{1}$ (Lemma \ref{carryPHS}). Therefore, an $r$-dimensional compact quaternionic Shimura variety $X$ determines $r$ plectic Jacobians $\{\mrm{J}_{\plectic}(X,j)\}_{j=1}^r$ which are abelian varieties (Proposition \ref{AVs}) and conjecturally are closely related to modular elliptic curves (Conjectures \ref{plecticODA} $\&$ \ref{plecticPARAM}). 
\begin{remark}
    When the quaternionic Shimura variety $X$ has odd dimension $r$, all middle degree cohomology classes are cuspidal. Thus, the real torus
   \[
   \mrm{H}^r(X,\bb{R})/\mrm{H}^r(X,\bb{Z})
   \]
   can be endowed with several complex structures: those arising from our definitions, and those considered by Weil (\cite{Weil} $\&$ \cite{Lieberman}, Section 3) and Griffiths (\cite{Griffiths}, Section 3). However, while Weil's and Griffiths' definitions of intermediate Jacobians exclude even cohomological degrees, our definition applies unchanged to the middle degree cohomology of even dimensional QSVs.
\end{remark}

 To add details to our discussion, let us consider a totally real number field $F$ of narrow class number one, and a non-split quaternion algebra $B/F$ with $\Sigma:=\{\nu_1,\dots,\nu_r\}$ as set of split Archimedean places. Recall that a quaternionic Shimura variety $X_B$, associated to $B/F$ and a choice of Eichler order, has a canonical model over the reflex field $\bb{Q}(\sum_{j=1}^r\nu_j(x)\lvert x\in F)\subset\bb{C}$. 
 The following conjecture aims at elucidating the relations between the plectic Jacobians of $X_B$. 
 
 \begin{conjecture}\label{ALGplecticJAC}
 There is an abelian variety $\mrm{J}_{\plectic}(X_B)$ defined over $F$ and canonical isomorphisms
 \[
\big(\mrm{J}_{\plectic}(X_B)\otimes_{F,\nu_j}\C\big)^\mrm{an}\cong\mrm{J}_{\plectic}(X_B,j)\qquad\forall\ j=1,\dots,r.
 \]
 \end{conjecture}
 \begin{remark}
Conjecture \ref{ALGplecticJAC} is well-known when $r=1$, i.e., whenever $X_B$ is a Shimura curve, while it is wide open for all $r\ge2$.  
 \end{remark}

 \noindent The Griffiths' style definition of plectic Jacobians allows us to define an exotic Abel--Jacobi maps with domain a subgroup of zero-cycles which we now describe.
We begin by recalling the complex uniformization of a QSV. Let us fix an isomorphism $\iota_j\colon B\otimes_{F,\nu_j}\bb{R}\overset{\sim}{\to}\mrm{M}_2(\bb{R})$ for every $\nu_j\in\Sigma$, then, if the level of the Eichler order is large enough, there is a torsion-free arithmetic subgroup $\Gamma\le B^\times/F^\times$ such that 
\[
X_B=\Gamma\backslash\cal{H}_\Sigma
\] 
is a complex manifold where $\Gamma$ acts on the product $\cal{H}_\Sigma=\prod_{\nu_j\in\Sigma}\cal{H}_{\nu_j}$ of Poincar\'e's upper-half planes via M\"obius transformations. For technical reasons (see Equation \eqref{keyvanishing}) our exotic Abel-Jacobi map is only defined for zero-cycles supported at ``generic'' points: 
denoting by $\cal{H}_{\nu_j}^\circ\subseteq \cal{H}_{\nu_j}$ the subset of those points  with trivial stabilizer in $\iota_j(\Gamma)\le \mrm{PGL}_2(\R)$, we can define
\[
\cal{H}^{\circ}_\Sigma:=\prod_{\nu_j\in\Sigma}\cal{H}_{\nu_j}^\circ\qquad\text{and}\qquad X^\circ_B:=\Gamma\backslash \cal{H}^{\circ}_\Sigma.
\]
 Note that this is not a serious restriction for arithmetic applications since the set  $X^\circ_B$ contains all CM points. 
 The free group $\bb{Z}\big[\cal{H}^{\circ}_\Sigma\big]$ of the product $\cal{H}^{\circ}_\Sigma$ is canonically isomorphic to $\otimes_{j=1}^r\mrm{Div}(\cal{H}_{\nu_j}^\circ)$ by mapping generators $[(\tau_1,\dots,\tau_r)]$ to elementary tensors $\otimes_{j=1}^r[\tau_j]$. If we denote by $\mrm{Div}^0(\cal{H}_{\nu_j}^\circ)$ the subgroup of degree-zero elements of $\mrm{Div}(\cal{H}_{\nu_j}^\circ)$, we can define  \emph{plectic zero-cycles} supported on $X^\circ_B$ by setting
\[
\Z_\plectic[X^\circ_B]:=\mrm{Im}\Big(\otimes_{j=1}^r\mrm{Div}^0(\cal{H}_{\nu_j}^\circ)\to\Z[X^\circ_B]\Big).
\]
Following Darmon--Logan (\cite{Darmon-Logan}) we consider the homomorphism 
\[
\int^r\colon\otimes_{j=1}^r\mrm{Div}^0(\cal{H}_{\nu_j}^\circ)\too\mrm{H}^r_\mrm{dR}(X_B)^\vee,\qquad \otimes_{j=1}^r\big([x_j]-[y_j]\big)\mapsto \int_{y_1}^{x_1}\hspace{-3mm}\cdots\hspace{-0.5mm} \int_{y_r}^{x_r}(-)
\]
mapping an elementary tensor to the linear functional computing a series of line integrals. We are now ready to state our second main theorem  which can be interpreted as a first step towards an Archimedean construction of plectic Heegner points. 
\begin{thmx}\label{thmB}
 The homomorphism $\int^r$ induces a well-defined Abel--Jacobi map
\[
\mrm{AJ}^j_\plectic\colon \Z_\plectic[X^\circ]\too \mrm{J}_\plectic(X_B,j)\qquad\forall\ j=1,\dots,r.
\] 
\end{thmx}

   In future work, we plan to perform numerical experiments to understand the feasibility of enlarging the domain of our exotic Abel--Jacobi maps to contain canonically defined zero-cycles supported on CM points.

\begin{acknowledgements}
I am grateful to G. Baldi, A. Doan, H. Esnault, L. Gehrmann, X. Guitart, D. Lilienfeldt, F. Lin, M. Masdeu, G. Oh, O. Rivero, A. Shnidman, J. Stelzig, M. Tamiozzo, and S. Zhang for answering my questions and participating in many enriching conversations. I would like to especially thank L. Gehrmann and M. Tamiozzo for generously sharing their time and for suggesting ways to improve this paper. 
The work on this article began while the author was a Simons Junior Fellow at Columbia University and was supported by the National Science Foundation under Grant No. DMS-1928930 while the author was in residence at the Mathematical Sciences Research Institute in Berkeley, California, during the Spring 2023 semester.
\end{acknowledgements}

\section{Rigidified complex manifolds}
\begin{definition}
	Let $U,V\subseteq\C^n$ be open subsets. We say that a holomorphic function $\phi\colon U\to V$ is \emph{rigid} if  there exist holomorphic functions $\{\phi^j\}_{j=1}^n$ in one variable such that
	\[
	\phi(u_1,\dots,u_n)=(\phi^1(u_1),\dots,\phi^n(u_n)).
	\]
\end{definition}

 Let $X$ be a Hausdorff topological space. An $n$-dimensional chart $(U,\phi)$ in $X$ consists of an open subset $U\subseteq X$ and an homeomorphism $\phi\colon U\to D$ onto an open subset $D\subseteq\C^n$. We say that two charts $(U,\phi), (V,\psi)$ are \emph{compatible} if either $U\cap V=\emptyset$, or the transition function
\[
\phi\circ\psi^{-1}\colon \psi(U\cap V)\to\phi(U\cap V)
\]
and its inverse are both rigid. 
A covering of $X$ consisting of pairwise compatible $n$-dimensional charts is called an $n$-dimensional \emph{rigidified atlas} of $X$. Moreover, two such atlases $\mathscr{A}_1, \mathscr{A}_2$ are called \emph{equivalent} if any two charts $(U,\phi)\in \mathscr{A}_1$ and $(V,\psi)\in \mathscr{A}_2$ are compatible. Finally, an equivalence class of $n$-dimensional rigidified atlases on $X$ is called an $n$-dimensional \emph{rigidified holomorphic structure} on $X$. It contains a maximal atlas which is the union of the atlases in the equivalence class.

\begin{definition}\label{rigidCM}
	An $n$-dimensional \emph{rigidified complex manifold} consists of a Hausdorff space $X$ with a countable basis, equipped with an $n$-dimensional rigidified holomorphic structure.
\end{definition}

\subsubsection{Examples.}  Any open subset $\Omega\subseteq \C^n$ has a natural structure of rigidified complex manifold given by the atlas
\[
\mathscr{A}=\{(U,\mathrm{id}_U)\ \vert\ U\ \text{open subset of}\ \Omega\}.
\]
The class of examples most relevant for our arithmetic applications consists of quotients $\Gamma\backslash\Omega$ of an open subset $\Omega\subseteq\C^n$ by a discrete group $\Gamma$.

\begin{lemma}\label{ExistExamples}
	Let $\Gamma$ be a discrete group acting on a connected open subset $\Omega\subseteq\C^n$. Suppose
	\begin{itemize}
		\item [$\bfcdot$] $\Gamma$ acts smoothly, freely and properly on $\Omega$,
		\item [$\bfcdot$] there exists a homomorphism $\Gamma\to \mrm{GL}_2(\C)^n$, $\gamma\mapsto(\gamma_1,\dots,\gamma_n)$, such that
		\[
		\gamma\cdot(x_1,\dots,x_n)=(\gamma_1(x_1),\dots,\gamma_n(x_n))\qquad \forall\gamma\in\Gamma,
		\]
	\end{itemize}
	then $\Gamma\backslash\Omega$ has a structure of rigidified complex manifold.
\end{lemma}
\begin{proof}
	Let $\pi:\Omega\to\Gamma\backslash\Omega$ be the quotient map and $\mathscr{A}$ an atlas in the canonical rigidified holomorphic structure of $\Omega$. We define a rigidified atlas $\mathscr{A}_\Gamma$ for  $\Gamma\backslash\Omega$ as the collection of all pairs $(\pi(U),\pi_{\lvert U}^{-1})$ such that  $(U,\mathrm{id}_U)$ belongs to $\mathscr{A}$  and $\pi_{\lvert U}:U\to\pi(U)$ is injective. First, as $\Gamma$ acts smoothly, freely and properly on $\Omega$ the quotient $\Gamma\backslash\Omega$ is a complex manifold. Then,  $\mathscr{A}_\Gamma$ is a rigidified atlas because its transition functions are given by the action of elements of the group, since $\Gamma$ is discrete.
\end{proof}

\begin{remark}
	Lemma \ref{ExistExamples} shows that complex tori and QSVs are natural examples of  rigidified complex manifolds. Moreover, we note here that the notion of rigidified complex manifolds could be generalized to include symplectic and unitary Shimura varieties over totally real number fields. 
\end{remark}

\subsection{Foliations}
In this section we explain why the tangent bundle  of a rigidified complex manifold admits a natural decomposition. For readers interested in the relation between split tangent bundles and product structures of the universal covering space, we refer to the articles \cite{Beauville}, \cite{Druel}.

\noindent Let $X$ be an $n$-dimensional rigidified complex manifold. For any index $j=1,\dots,n$ we define the $j$-th sub-vector bundle $T_X^j$ of the tangent bundle $T_X$ of $X$ as follows. Let $\scr{U}=\{U_k\}_k$ be an open covering of $X$, and set $U_{k,\ell}:=U_k\cap U_\ell$. Then the rank $2$ real vector bundle $T_X^j$ is covered by open sets $\{U_k\times\R^{2}\}_k$ and the transition morphism between 
\[
U_{k,\ell}\times \R^2\subseteq U_k\times\R^2\qquad \text{and}\qquad U_{k,\ell}\times \R^2\subseteq U_\ell\times\R^2
\] 
 is given by $(u,v)\mapsto (u,d\phi_{ik}^j(u)(v))$, where $\phi_{k,\ell}=(\phi_{k,\ell}^1,\dots,\phi_{k,\ell}^n)$ is the rigid transition map between $\phi_k(U_{k,\ell})\subseteq\C^n$ and $\phi_\ell(U_{k,\ell})\subseteq\C^n$, and $d\phi_{k,\ell}^j(u)\colon\R^2\to\R^2$ denotes the Jacobian matrix of $\phi_{k,\ell}^j$ at the point $u$. We note that there is a direct sum decomposition of the tangent bundle
	\[
	T_X=\bigoplus_{j=1}^nT_X^j.
	\]
Since $X$ is a complex manifold, each vector bundle $T_X^j$ is equipped with an almost complex structure $I_j$. Therefore, there is a decomposition of the extension of scalars 
\[
T_X^j\otimes_\R\C=T_X^{1_j,0_j}\oplus T_X^{0_j,1_j}
\]
where $T_X^{1_j,0_j}$ (resp. $T_X^{0_j,1_j}$) is the sub-bundles of $T_X^j\otimes_\R\C$ onto which the involution $I_j$ acts with eigenvalue $i$ (resp. $-i$).
	\begin{remark}
		Let $(z_1,\dots,z_n)$ be local complex coordinates on an open subset $U\subseteq X$ trivializing $T^j_X$, and denote by $\scr{C}^\infty(U)$ the ring of smooth $\C$-valued functions on $U$. If we write $z_j=x_j+iy_j$, then we can explicitly describe smooth section as
		\[
		T_X^j\otimes_\R\C(U)=\scr{C}^\infty(U)\cdot\frac{\partial}{\partial x_j}\oplus\scr{C}^\infty(U)\cdot\frac{\partial}{\partial y_j}.
		\]
Moreover, smooth sections of $T_X^{1_j,0_j}$ and $T_X^{0_j,1_j}$ over $U$ are free of rank one over $\scr{C}^\infty(U)$ with respective basis elements 
		\[
		\frac{\partial}{\partial z_j}:=\frac{1}{2}\left(\frac{\partial}{\partial x_j}-i\frac{\partial}{\partial y_j}\right)\qquad\text{and}\qquad \frac{\partial}{\partial \overline{z}_j}:=\frac{1}{2}\left(\frac{\partial}{\partial x_j}+i\frac{\partial}{\partial y_j}\right).
		\]
	\end{remark}



\subsubsection{Refined types of differential forms.}
The classical Hodge decomposition of differential form on complex manifolds induces a factorization of the exterior differential $\mrm{d}_X$ into a holomorphic and an anti-holomorphic component $\mrm{d}_X=\partial_X+\overline{\partial}_X$. The decomposition of the tangent bundle of rigidified complex manifolds further refines the types of differential forms and a fortiori the factorization of the exterior differential.

\begin{definition}
	Let $X$ be an $n$-dimensional rigidified complex manifold. For $j\in\{1,\dots,n\}$ set 
	\[
	\cal{A}_X^{1_j,0_j}:=\mrm{Hom}_{\C}\Big(T_X^{1_j,0_j}, \C\Big)\qquad\text{and}\qquad \cal{A}_X^{0_j,1_j}:=\mrm{Hom}_{\C}\Big(T_X^{0_j,1_j}, \C\Big).
	\]
	Then, for an ordered pair $(\alpha,\beta)$ of elements in $\{0,1\}^n$, we define the vector bundle of $\C$-valued smooth differential forms of type $(\alpha,\beta)$ by
	\begin{equation}
	\cal{A}_X^{\alpha,\beta}:=\bigotimes_{\alpha_j=1}\cal{A}_X^{1_j,0_j}\otimes\bigotimes_{\beta_j=1} \cal{A}_X^{0_j,1_j}.
	\end{equation}
\end{definition}

\begin{remark}
Exterior powers of smooth differential forms $\cal{A}_X:=\mrm{Hom}_{\C}\big(T_X, \C\big)$ admit  direct sum decomposition of the form
	\[
	\wedge^k\cal{A}_X=\bigoplus_{\lvert \alpha+\beta\rvert=k}\cal{A}_X^{\alpha,\beta}.
	\]
 \end{remark}
 \noindent Let $j\in\{1,\dots,n\}$ and $\alpha,\beta\in\{0,1\}^n$. If $\alpha_j=0$ there is a differential operator $\xi_j\colon\cal{A}_X^{\alpha,\beta}\to \cal{A}_X^{\alpha+1_j,\beta}$ defined by the diagram
\begin{equation}
\xymatrix{
\cal{A}_X^{\alpha,\beta}\ar@{.>}[r]^{\xi_j}\ar@{^{(}->}[d] &\cal{A}_X^{\alpha+1_j,\beta}\\ \wedge^{\lvert \alpha+\beta\rvert}\cal{A}_X\ar[r]^{\mrm{d}_X}&\wedge^{\lvert \alpha+\beta\rvert+1}\cal{A}_X.\ar@{->>}[u]
}\end{equation}
 If we simply set $\xi_j\colon\cal{A}_X^{\alpha,\beta}\to \{0\}$ when $\alpha_j=1$, we can write $\partial_X=\sum_{j=1}^n\xi_j$. On a local chart with coordinates $(z_1,\dots,z_n)$, any differential form $\omega\in \cal{A}_X^{\alpha,\beta}$ can be written as \[
 \omega=f\cdot\mrm{d}z_\alpha\wedge\mrm{d}\overline{z}_\beta,\qquad\text{where}\qquad\mrm{d}z_\alpha=\wedge_{\{j:\hspace{0.5mm}\alpha_j=1\}}\mrm{d}z_j,\quad \mrm{d}\overline{z}_\beta=\wedge_{\{j:\hspace{0.5mm} \beta_j=1\}}\mrm{d}\overline{z}_j,
 \]
 and the differential operator $\xi_j$ is given by the formula
	\begin{equation}
	\xi_j(\omega)=\frac{\partial}{\partial z_j}f\cdot\mrm{d}z_j\wedge\mrm{d}z_\alpha\wedge\mrm{d}\overline{z}_\beta.
	\end{equation}

\section{Refined Hodge decomposition}
To promote the refined decomposition of differential forms into a refinement of the Hodge decomposition of de Rham cohomology, it is necessary to understand when the Laplacian operator associated to a K\"ahler metric respects the refined types of differential forms. The next definition singles out a sufficient condition. 
 
\begin{definition}
We say that a hermitian metric $\mathrm{d}s^2$ on an $n$-dimensional rigidified complex manifold $X$ is \emph{distinctive K\"ahler} if in a neighborhood of every point $x\in X$ there is a holomorphic coordinate system $(z_1,\dots,z_n)$ and a unitary coframe $\varphi_1,\dots,\varphi_n$ for the metric, such that 
\[
\varphi_j=f_j(z_j)\cdot\mathrm{d}z_j\qquad\&\qquad\frac{\partial}{\partial \overline{z}_j}f_j(x)=0\qquad\forall\ j=1,\dots,n.
\]
Therefore,  $\varphi_j$ is a differential form of type $(1_j,0)$ which satisfies $ \frac{\partial}{\partial z_k}f_j\equiv0\equiv \frac{\partial}{\partial \overline{z}_k}f_j$  $\ \ \forall\ k\not=j$.
\end{definition}
\begin{remark}
A distinctive K\"ahler metric is also K\"ahler. Indeed, one of the equivalent conditions for a metric on an $n$-dimensional complex manifold to be K\"ahler is the existence, for any $x\in X$, of a unitary coframe $\varphi_1,\dots,\varphi_n$ in a neighborhood of $x$ such that $\mrm{d}_X\varphi_j(x)=0$ for every $j=1,\dots,n$.
\end{remark}
\begin{definition}\label{def RIG}
	A \emph{rigidified K\"ahler manifold} is a rigidified complex manifold admitting a distinctive K\"ahler metric.
\end{definition}
  \noindent We now show that rigidified K\"ahler manifolds exist in nature. For every $j=1,\dots,n$ consider a connected open subset $\Omega_j\subseteq\C$ admitting a K\"ahler metric $\mathrm{d}s^2_i$. Write $\Omega=\prod_{j=1}^n\Omega_j$ and suppose a discrete group  $\Gamma$ acts  on it satisfying the assumptions of Lemma \ref{ExistExamples}.
\begin{lemma}\label{metricRIG}
	If the $\Gamma$-action preserves the metric $\mathrm{d}s^2=\sum_j\mathrm{d}s^2_j$, then $\Gamma\backslash\Omega$ has a natural structure of rigidified K\"ahler manifold.
\end{lemma}
\begin{proof}
	Lemma \ref{ExistExamples} shows that $\Gamma\backslash\Omega$ has a structure of rigidified complex manifold. We just have to prove that the metric $\mathrm{d}s^2=\sum_j\mathrm{d}s^2_j$ induces a distinctive K\"ahler metric on the quotient. For any point $x\in \Gamma\backslash\Omega$ we choose a lift $\widetilde{x}=(x_1,\dots,x_n)\in\Omega$ and open neighborhoods $U_j\subseteq\Omega_j$ of $x_j$ such that the projection map $\pi\colon \Omega\to \Gamma\backslash\Omega$ is injective when restricted to $U=\prod_jU_j$. As each $(\Omega_j, \mathrm{d}s^2_j)$ is a K\"ahler manifold, we can suppose there is a holomorphic coordinate $z_j$ on $U_j$ and a unitary coframe $\varphi_j=f_j(z_j)\cdot\mathrm{d}z_j$ for the metric $\mathrm{d}s^2_j$ satisfying $\mrm{d}\varphi_j(x_j)=0$. Then, the collection of all the pull-backs of the $\varphi_j$'s to the product $U$ gives the sought-after unitary coframe.
\end{proof}

\begin{remark}
Lemma \ref{metricRIG} implies that complex tori and QSVs admit natural structures of rigidified K\"ahler manifolds.	
\end{remark}


\subsection{Refined Hodge identities}

Following (\cite{GriffithsHarris}, Ch. 0, Sect. 6), we recall the definitions of adjoint differential operatiors and  apply it to the special case of rigidified K\"ahler manifolds. For a connected, compact, $n$-dimensional rigidified K\"ahler manifold $X$, we fix a distinctive K\"ahler metric $\mrm{d}s^2$ with associated $(1,1)$-form $\omega$ locally given in a unitary coframe by
\[
\omega=\frac{i}{2}\sum_{j=1}^n\varphi_j\wedge\overline{\varphi}_j.
\]
The distinctive K\"ahler metric induces a hermitian metric on the space of differential forms which can be used in combination with the volume form $\omega^n$ to define the inner product
\[
\langle\psi,\eta\rangle:=\int_X\big(\psi(x),\eta(x)\big)\frac{\omega^n}{n!}\qquad\forall\ \psi,\eta\in A^{p,q}(X).
\]
Then, the adjoint  differential operator $\xi_j^\star\colon A^{p,q}(X)\to A^{p-1,q}(X)$ is defined by the formula
\[
\langle\xi_j^\star\psi,\hspace{0.5mm}\eta\rangle=\langle\psi,\hspace{0.5mm}\xi_j\eta\rangle\qquad \forall\ \eta\in A^{p-1,q}(X).
\]


\begin{proposition}\label{refined Hodge identities}
	If $X$ is a rigidified K\"ahler manifold, then
	\[
	\xi_j\cdot\xi_k^\star+\xi_k^\star\cdot\xi_j=0\qquad \forall\ j\not=k.
	\]
\end{proposition}
\begin{proof}
	We adapt the proof of the classical Hodge identities (\cite{GriffithsHarris}, page 111) and begin by verifying the claim for $\C^n$ with the Euclidean metric. The idea is to write the differential operators $\xi_j$'s as a composition of simpler operators. For each index $j=1,\dots,n$ we consider the operator $e_j\colon A_c^{a,b}(\C^n)\to A_c^{a+1_j,b}(\C^n)$ 
	on compactly supported forms defined by 
	\[
	e_j(\psi)=\mathrm{d}z_j\wedge\psi.
	\]
	Let $e_j^\star$ 
	denote the adjoint of $e_j$, and 
	note that the operators $e_j,e_j^\star$ 
	are $\mathscr{C}^\infty(\C^n)$-linear. By a direct computation one verifies that 
	\[
	e_j^\star\cdot e_k+e_k\cdot e^\star_j=0\qquad \forall\ j\not=k. 
	\]
	For $j=1,\dots,n$ we also consider the operator $\partial_j$ 
	 on $A^{a,b}_c(\C^n)$ defined by 
	\[
	\partial_j\big(f\cdot \mathrm{d}z_a\wedge\mathrm{d}\overline{z}_b\big)=\frac{\partial}{\partial z_j}f\cdot \mathrm{d}z_a\wedge\mathrm{d}\overline{z}_b.
	\]
	The operators $\partial_j$'s 
	commute with each other, with all $e_k,e^\star_k$'s, and 
	satisfy $\partial_j^\star=-\overline{\partial}_j$, i.e., 
 \[
 \partial_j^\star\big(f\cdot \mathrm{d}z_a\wedge\mathrm{d}\overline{z}_b\big)=-\frac{\partial}{\partial \overline{z}_j}f\cdot \mathrm{d}z_a\wedge\mathrm{d}\overline{z}_b.
 \]
	We can then write
	\[
	\xi_j=\partial_j\cdot e_j,\qquad \xi_j^\star=-\overline{\partial}_j\cdot e^\star_j.
	\]
For $j\not=k$, the straightforward computation
		\[\begin{split}
		\xi_j\cdot\xi_k^\star&=\partial_j\cdot e_j\cdot(-\overline{\partial}_k\cdot e^\star_k)\\
		&=\overline{\partial}_k\cdot \partial_j\cdot e^\star_k\cdot e_j\\
		&=-\xi_k^\star\cdot\xi_j
		\end{split}\]
  proves the proposition for $\bb{C}^n$.
We claim that the computations with the Euclidean metric suffice to the deduce the result for any rigidified K\"ahler manifold $X$. Indeed, recall that by assumption, in a neighborhood of any point $x\in X$, we can find a holomorphic coordinate system $(z_1,\dots,z_n)$ and a unitary coframe $\{\varphi_j=f_j(z_j)\cdot\mathrm{d}z_j\}_{j=1}^n$ for the metric, such that 
\[
\frac{\partial}{\partial z_k}f_j\equiv0\equiv \frac{\partial}{\partial \overline{z}_k}f_j\qquad\&\qquad\frac{\partial}{\partial \overline{z}_j}f_j(x)=0\qquad\forall\ j=1,\dots,n,\ \ \forall\ k\not=j.
\]
In general, the operator $\xi_j$ equals $\partial_j\cdot e_j$ up to terms that only involve the first order derivative of the function $f_j$. Then, as the operators $e_j, e^\star_j$'s are linear with respect to the algebra of $\mathscr{C}^\infty$-functions, we deduce that  $\xi_j\cdot\xi_k^\star+\xi_k^\star\cdot\xi_j$ for $k\not=j$ coincides at $x\in X$ with the zero-operator up to terms that involve first derivatives, products of first derivatives, and mixed second derivatives of the functions $f_j$ and $f_k$. The terms containing the first order derivatives or their products vanish at $x\in X$ because of the usual K\"ahler condition, while the terms containing the mixed partial derivatives vanish identically because each function $f_j$ depends on a single holomorphic coordinate.

\end{proof}

\noindent Recall that given an operator $P$ on differential forms, the associated Laplacian is the degree zero operator given by the formula $\Delta_P=P\cdot P^\star+P^\star\cdot P$.  For a K\"ahler manifold $X$ we set $\Delta_\mrm{d}=\Delta_{\mrm{d}_X}$.
\begin{corollary}\label{respect}
	If $X$ is an $n$-dimensional rigidified K\"ahler manifold, then 
	\[
	\Delta_\mrm{d}=\frac{1}{2}\sum_{j=1}^n\Delta_{\xi_j}.
	\]
	In particular, the Laplacian operator $\Delta_\mrm{d}$ respects differential forms of refined type $(\alpha,\beta)$. 
\end{corollary}
\begin{proof}
Since $X$ is a K\"ahler manifold we know that
\[
\Delta_\mrm{d}=2\Delta_{\partial}.
\]
The claim then follows from a direct computation using Proposition \ref{refined Hodge identities}:
	\[\begin{split}
		\Delta_\partial&=\partial\cdot\partial^\star+\partial^\star\cdot\partial\\
		&=\Big(\sum_j\xi_j\Big)\cdot\Big(\sum_k\xi_k^\star\Big)+\Big(\sum_k\xi_k^\star\Big)\cdot \Big(\sum_j\xi_j\Big)\\
		&=\sum_{j,k}\xi_j\cdot\xi_k^\star+\sum_{j,k}\xi_k^\star\cdot\xi_j\\
		&=\sum_j\big(\xi_j\cdot\xi_j^\star+\xi_j^\star\cdot\xi_j\big)+\sum_{j\not=k}\big(\xi_j\cdot\xi_k^\star+\xi_k^\star\cdot\xi_j\big)\\
		&=\sum_j\Delta_{\xi_j}.
	\end{split}\]
\end{proof}

\noindent We denote by $\cal{H}^{\alpha,\beta}(X)$ the space of harmonic differential forms of refined type $(\alpha,\beta)$, i.e.,  
\[
\cal{H}^{\alpha,\beta}(X):=\big\{\psi\in A^{\alpha,\beta}(X)\ \mid\ \Delta_\mrm{d}\psi=0\big\}.
\]
Note that $\cal{H}^{\alpha,\beta}(X)=\overline{\cal{H}^{\beta,\alpha}(X)}$ since the Laplacian operator $\Delta_\mrm{d}$ is real.
\begin{corollary}\label{plectic Hodge structure}
	The de Rham cohomology of any compact rigidified K\"ahler manifold $X$  admits a canonical direct sum decomposition
	\[
	\mathrm{H}^k_{\mathrm{dR}}(X/\C)=\bigoplus_{\lvert \alpha +\beta\rvert=k}\cal{H}^{\alpha,\beta}(X),\qquad \alpha,\beta\in \{0,1\}^{\mrm{dim}X}
	\]
	in terms of harmonic differential forms of refined types.
\end{corollary}
\begin{proof}
	The usual Hodge decomposition describes the de Rham cohomology of K\"ahler manifolds in terms of harmonic differential forms. As the Laplacian operator $\Delta_\mrm{d}$ respects differential forms of refined type $(\alpha,\beta)$ when $X$ is rigidified K\"ahler (Corollary \ref{respect}), the claim follows.
\end{proof}
\begin{remark}
When $X$ is a compact quaternionic Shimura variety, Corollary \ref{plectic Hodge structure} can also be deduced from results proved by Matsushima and Shimura in \cite{Matsu-Shimu}. 
\end{remark}

\begin{lemma}
	The refined Hodge decomposition does not depend on the choice of a distinctive K\"ahler metric, i.e., it only depends on the rigidified complex structure.
\end{lemma}
\begin{proof}
Voisin's proof (\cite{Voisin}, Proposition 6.11) of the independence of the Hodge decomposition from the choice of K\"ahler metric can be easily adapted to our setting. We aim to show	 that the subspace of de Rham cohomology classes which are representable by a closed form of refined type $(\alpha,\beta)$ coincides with $\cal{H}^{\alpha,\beta}(X)$. 
Let $\psi$ be a closed form of refined type $(\alpha,\beta)$. We can write in a unique way $\psi=\eta+\Delta_\mrm{d}\zeta$ with $\eta$ harmonic. Since $\Delta_\mrm{d}$ respects refined types (Proposition \ref{respect}), we can further suppose that both $\eta$ and $\zeta$ are of refined type $(\alpha,\beta)$. As $\ker(\Delta_\mrm{d})\subseteq\ker(\mrm{d})$, we see that $\mrm{d}\eta=0$ and compute that $\Delta_\mrm{d}\zeta$ is closed:
\[
\mrm{d}\circ\Delta_\mrm{d}\zeta=\mrm{d}\eta+\mrm{d}\circ\Delta_\mrm{d}\zeta=\mrm{d}\psi=0.
\]
Since $\mrm{d}^2=0$, we further deduce that $\mrm{d}^\star \mrm{d}\zeta\in \ker(\mrm{d})$.  As the intersection $\ker (\mrm{d})\cap\text{Im}(\mrm{d}^\star)$ is trivial, we obtain $\mrm{d}^\star \mrm{d}\zeta=0$. Hence, $\psi=\eta+\mrm{d}\mrm{d}^\star\zeta$, i.e., $\psi$ and $\eta$ are in the same de Rham class.
\end{proof}

\noindent In other words, we have shown that the cohomology of a compact rigidified K\"ahler manifold is endowed with a canonical plectic Hodge structure refining its classical Hodge structure.

\subsection{Functoriality of plectic Hodge structures}\label{Functoriality}

\begin{definition}
Let $X, Y$ be $n$-dimensional rigidified complex manifolds. A morphism from $X$ to $Y$ is function $\varphi\colon X\to Y$ such that for every point $x\in X$ there is a chart $(U,\phi)$ on $X$ with $x\in U$ and a chart $(V,\psi)$ on $Y$ with $\varphi(x)\in V$ such that $\varphi(U)\subseteq V$ and the composition $\psi\circ\varphi\circ\phi^{-1}$ is a rigid holomorphic function.
\end{definition}
\noindent What follows is a simple adaptation of \cite[Section 7.3.2]{Voisin}.
Let $X,Y$ be $n$-dimensional rigidified K\"ahler manifolds and $\varphi\colon X\to Y$ a morphism of rigidified complex manifolds. Using the de Rham cohomology description of the pullback map, it is clear that $\varphi^*\colon \mrm{H}^k(Y,\Z)\to\mrm{H}^k(X,\Z)$ is a morphism of $n$-plectic Hodge structures of degree $0$. 

\begin{lemma}
    The pushforward $\varphi_*\colon\mrm{H}^k(X,\Z)\to\mrm{H}^k(Y,\Z)$ is a morphism of $n$-plectic Hodge structures of degree $0$ for every $k\ge0$.
\end{lemma}
\begin{proof}
Recall that Poincar\'e duality $\langle\hspace{1mm},\hspace{0.5mm}\rangle_X\colon \colon\mrm{H}^k(X,\Z)\times \colon\mrm{H}^{2n-k}(X,\Z)\to\Z$ gives an isomorphism $\mrm{H}^k(X,\Z)\cong \mrm{H}^{2n-k}(X,\Z)^\vee$ between the maximal torsion-free quotients of the cohomology groups which can be used to characterize the pushforward $\varphi_*$ with the equation
    \[ \big\langle\varphi_*\kappa,\eta\big\rangle_Y=\big\langle\kappa,\varphi^*\eta\big\rangle_X \qquad\forall\ \kappa\in\mrm{H}^k(X,\Z),\ \eta\in\mrm{H}^{2n-k}(Y,\Z).
    \]
    Suppose now that $\kappa\in\cal{H}^{\alpha,\beta}(X)$, then to prove the lemma we need to show that $\varphi_*\kappa\in \cal{H}^{\alpha,\beta}(Y)$. This is a direct consequence of equality \eqref{Perpendicular} below and the fact that the pullback map $\varphi^*$ is a morphism of plectic Hodge structures of weight $0$. 
    
    \noindent For $\alpha\in\{0,1\}^n$ denote by $\alpha^c\in\{0,1\}^n$ the unique element such that $\alpha+\alpha^c=\underline{1}$. We claim that for any $(\alpha,\beta)$ satisfying $\lvert\alpha+\beta\rvert=k$ we have
    \begin{equation}\label{Perpendicular}
\cal{H}^{\alpha,\beta}=\left(\bigoplus_{(\gamma,\delta)\not=(\alpha^c,\beta^c)}\cal{H}^{\gamma,\delta}\right)^\perp
    \end{equation}
    where the pairs $(\gamma,\delta)$ also satisfy $\lvert\gamma+\delta\rvert=2n-k$ and the orthogonality is taken with respect to the Poincar\'e pairing. 
    To prove the claim we note that $\cal{H}^{\alpha,\beta}$ is contained in the right hand side (RHS), and the dimensions of the two spaces coincide
    \begin{equation}\label{comparingDIMS}
    \mrm{dim}\hspace{0.5mm} \text{RHS}=\mrm{dim}\hspace{0.5mm} \cal{H}^{\alpha^c,\beta^c}=\mrm{dim}\hspace{0.5mm} \cal{H}^{\alpha,\beta}.
    \end{equation}
Note that the first equality of \eqref{comparingDIMS} follows from the definitions, while the second comes from the isomorphism $\star\colon \cal{H}^{\alpha,\beta}\cong \cal{H}^{\alpha^c,\beta^c}$ given by the Hodge star operator induced by the \emph{distinguished K\"ahler metric}. We refer to \cite[page 82]{GriffithsHarris} for the definition of Hodge's star operator and to \cite[pages 101-102]{GriffithsHarris} for the relevant properties.  
\end{proof}

\section{Algebraicity of complex tori with real multiplication}
While the content of this section is well-known to experts, we decided to include it for the convenience of the reader.
Recall that for any complex torus $T=V/\Lambda$ there is an injective ring homomorphism 
\[
\mrm{End}(T)\hookrightarrow \mrm{End}_\Z(\Lambda).
\]
In particular, the ring $\mrm{End}(T)$ is torsion-free and finitely generated as a $\Z$-module.
\begin{definition}
   A complex torus $T$ has real multiplication if there exists a totally real field $L$ with $[L:\bb{Q}]=\mrm{dim}T$, and a ring homomorphism $\theta\colon L\to\mrm{End}(T)_\bb{Q}$ such that $\theta(1)=\mrm{id}_T$. We say that $T$ has real multiplication by an order $\cal{O}$ if $T$ has real multiplication and $\cal{O}:=\theta^{-1}(\mrm{End}(T))$.
\end{definition}

\begin{lemma}\label{RMbyMAX}
Any complex torus $T$ with real multiplication is isogenous to a complex torus $T'$ with real multiplication by the ring of integers of a totally real number field.
\end{lemma}
\begin{proof}
    Let $T=V/\Lambda$ be a complex torus with real multiplication given by $\theta\colon L\to \mrm{End}(T)_\bb{Q}$. As $\cal{O}:=\theta^{-1}(\mrm{End}(T))$ is an order in $L$, it has finite index in the ring of integers $\cal{O}_L$. Write $n=[\cal{O}_L:\cal{O}]$ and consider
    \[
\Lambda':=\bigcup_{x\in\cal{O}_L/\cal{O}}x\Lambda.
    \]
    We have $\Lambda\subseteq\Lambda'\subseteq\frac{1}{n}\Lambda$, hence $\Lambda'$ a lattice in $V$ commensurable with $\Lambda$. Moreover, $T':=V/\Lambda'$ is a complex torus with real multiplication by $\cal{O}_L$, isogenous to $T$. 
\end{proof}

\begin{proposition}{(\cite{MayNotes}, Theorem 7.2)}\label{PROJoverDED}
Let $R$ be a Dedekind ring. For a finitely generated $R$-module $M$  the following are equivalent
\begin{itemize}
    \item [$\bfcdot$] $M$ is $R$-projective,
     \item [$\bfcdot$] $M$ is $R$-flat,
      \item [$\bfcdot$]$M$ is $R$-torsion-free.
\end{itemize}
Moreover, any finitely generated and torsion-free $R$-module $M$ is isomorphic to $R^{n-1}\oplus \frak{A}$ for some fractional ideal $\frak{A}$ of $R$ and $n=\mrm{rk}_RM$.
\end{proposition}

\begin{corollary}\label{propertyHOMOLOGY}
    Let $T=V/\Lambda$ be a complex torus with real multiplication by $\cal{O}_L$, the ring  of integers of a totally real number field. Then, $\Lambda$ is $\cal{O}_L$-projective of rank two. 
\end{corollary}
\begin{proof}
By considering the inclusion $\mrm{End}(T)\hookrightarrow \mrm{End}_\Z(\Lambda)\cong \mrm{M}_{\mrm{rk}\Lambda}(\Z)$, we see that the non-zero elements of $\cal{O}_L$ act on $\Lambda$ via elements of $\mrm{M}_{\mrm{rk}\Lambda}(\Z)\cap\mrm{GL}_{\mrm{rk}\Lambda}(\bb{Q})$. Hence, $\Lambda$ is $\cal{O}_L$-torsion-free. Now, \cite[Corollary 2.6]{GorenHMVs} tells us that $V=\Lambda\otimes_\Z\R$ is a free $\cal{O}_L\otimes_\Z\C$-module of rank one. Therefore, we compute that 
\[
\mrm{rk}_{\cal{O}_L}\hspace{0.5mm}\Lambda=\mrm{rk}_{\cal{O}_L\otimes_\bb{Z}\R}\hspace{0.5mm}V=2\cdot \mrm{rk}_{\cal{O}_L\otimes_\bb{Z}\C}\hspace{0.5mm}V=2.
\]
 Proposition \ref{PROJoverDED} then finishes the proof.
\end{proof}

The following result can also be found in \cite[Ch. IX, Lemma 1.4]{Geer}.
\begin{theorem}\label{RMTisAV}
Every complex torus $T$ with real multiplication is an abelian variety.
\end{theorem}
\begin{proof}
     Up to isogeny (\cite{ComplexTori}, Ch. 2, Proposition 1.1), we can assume that $T=V/\Lambda$ has real multiplication by a maximal order $\cal{O}_L$ (Lemma \ref{RMbyMAX}), then $\Lambda\cong \cal{O}_L\oplus \frak{A}$ for some fractional ideal $\frak{A}$ (Corollary \ref{propertyHOMOLOGY} $\&$ Proposition \ref{PROJoverDED}). Set $\Sigma=\mrm{Hom}_\bb{Q}(L,\C)$. As $V$ is a free $\cal{O}_L\otimes_\Z\C$-module of rank one, it admits a decomposition into $1$-dimensional $\C$-subspaces
    \[
    V=\bigoplus_{\sigma\in\Sigma} V^\sigma
    \]
such that an element $x\in\cal{O}_L$ acts on $V^\sigma$ as multiplication by $\sigma(x)$. From a choice of isomorphism $\phi\colon  \Lambda\overset{\sim}{\to} \cal{O}_L\oplus \frak{A}$ of $\cal{O}_L$-modules, we obtain an $\R$-linear isomorphism $(\phi\otimes1)^\sigma\colon V^\sigma\overset{\sim}{\to} \R^2$ for every $\sigma\in\Sigma$ by extending scalars to $\R$. Moreover, transporting the complex structure from $V^\sigma$ to $\R^2$, we can promote them to $\C$-linear identifications $(\phi\otimes1)^\sigma\colon V^\sigma\overset{\sim}{\to}\C$. Thus, there are $\lambda, \mu\in(\C^\times)_\Sigma$ such that $\lambda_\sigma,\mu_\sigma\in\C^\times$ are $\R$-linearly independent $\forall\sigma\in\Sigma$ and
\[
\mu^{-1}(\phi\otimes1)\Lambda=\cal{O}_L\cdot (\lambda\mu^{-1})+ \frak{A}\cdot\underline{1}.
\]
As $L$ is dense in $L\otimes_\Q\R$, the  composition $L^\times\to(L\otimes_\Q\R)^\times\to \{\pm1\}_{\Sigma}$ is surjective, and we can choose $x\in L$ such that every component of $x(\lambda\mu^{-1})\in\C_\Sigma$ has positive imaginary part. Hence, we have shown the existence of $z\in\cal{H}^\Sigma$, a fractional ideal $\frak{B}$ of $\cal{O}_L$, and an isomorphism of complex tori
\[
T\cong \C_\Sigma/\Lambda_z \quad\text{where}\quad \Lambda_z=\cal{O}_L\cdot z+\frak{B}.
\]
We deduce that $T$ is an abelian variety using \cite[Corollary 2.10]{GorenHMVs}. 
\end{proof}
\begin{remark}
    Every complex manifold has at most one algebraic structure because the analytification functor is fully faithful (\cite{SGA1}, Corollaire 4.5). 
\end{remark}
\begin{definition}
    A rational Hodge structure $H_\bb{Q}$ of weight $1$  has real multiplication if there is a totally real field $L$ with $2[L:\bb{Q}]=\mrm{dim}_\Q H_\bb{Q}$, and a ring homomorphism $\theta\colon L\to\mrm{End}(H_\bb{Q})$ such that $\theta(1)=\mrm{id}_{H_\bb{Q}}$.
\end{definition}
\begin{corollary}\label{ALGforHSwithRM}
The isogeny class of complex tori associated to an effective rational Hodge structure of weight $1$ with real multiplication consists of abelian varieties.
\end{corollary}
\begin{proof}
    Let $H_\bb{Q}$ be an effective rational Hodge structure of weight $1$ with real multiplication. Directly from the definitions, the Jacobian associated to any lattice $\Lambda\subset H_\bb{Q}$ is a complex torus with real multiplication. The claim then follows from  Theorem \ref{RMTisAV}.
\end{proof}

\section{Archimedean plectic Jacobians}
Recall that an $n$-plectic Hodge structure $H$ is called effective of weight $\underline{1}\in\bb{Z}^n$ if 
\[
H^{\alpha,\beta}\not=0\quad\implies\quad \alpha,\beta\in\bb{N}^n\quad\&\quad\alpha+\beta=\underline{1}.
\]
We can think of an $n$-plectic Hodge structure of weight  $\underline{1}\in\bb{Z}^n$ as a collection of $n$ classical Hodge structure of weight $1$ on the same underlying module by setting for any $j=1,\dots, n$
\begin{equation}\label{functorTOclassic}
F^{1_j}=F^{1_j}(H):=\bigoplus_{\alpha_j\ge 1}H^{\alpha,\beta}.
\end{equation}
Therefore, for any $j=1,\dots,n$ the plectic Jacobian $\mrm{J}_{\plectic}(H,j)$ associated to an effective $n$-plectic Hodge structure $H$  of weight $\underline{1}\in\bb{Z}^{n}$ is the complex torus defined by 
	\[
	\mrm{J}_{\plectic}(H,j):=H\backslash (H\otimes_\Z\C)/F^{1_j}.
	\]
Note that if $H_j=(H,F^{1_j})$ denotes the $j$-th Hodge structure of weight $1$, we see that points of plectic Jacobians parametrize extensions of mixed Hodge structures (\cite{MHS}, Example 3.34)
\[
\mrm{J}_{\plectic}(H,j)=\mrm{Ext}_{\mrm{MHS}}(\Z(-1),H_j).
\]

\subsubsection{Example: products of complex tori.}\label{ExTori}
Recall that the Jacobian of the weight one Hodge structure appearing in the first cohomology group  $\mrm{H}^1(T,\Z)$ of a complex torus $T$ recovers the dual torus $T^\vee$   (\cite{Voisin}, Section 7.2.2). If we consider a product $X=T_1\times\cdots\times T_n$ of complex tori, then the tensor product 
\[
\mrm{H}^n(X,\Z)_{\underline{1}}:=\bigotimes_{j=1}^n\mrm{H}^1(T_j,\Z)
\]
is an effective $n$-plectic Hodge structure of weight $\underline{1}\in\bb{Z}^n$ satisfying
\[
F^{1_j}\mrm{H}^n(X,\C)_{\underline{1}}=F^1\mrm{H}^1(T_j,\C)\otimes_\Z\bigotimes_{k\not=j}\mrm{H}^1(T_k,\Z),
\]
and whose plectic Jacobians can be explicitly described.
For any index $j=1,\dots,n$ let us define the plectic Jacobian $\mrm{J}_\plectic(X,j)$ as the plectic Jacobian associated to the plectic Hodge structure $\mrm{H}^n(X,\Z)_{\underline{1}}$. A direct calculation shows that
\[
 \mrm{J}_\plectic(X,j)\cong T_j^\vee\otimes_\Z\bigotimes_{k\not=j}\mrm{H}^1(T_k,\Z).
\]
In particular, if $T_j$ is an abelian variety defined over a number field, then $\mrm{J}_\plectic(X,j)$ is also an abelian variety with a distinguished model over the same number field.

\subsection{Compact quaternionic Shimura varieties}
Let $F$ be a totally real number field of degree $d=[F:\Q]$ and, to simplify the exposition, of narrow class number one. Let $B/F$ be a non-split quaternion $F$-algebra, denote by $\Sigma=\{\nu_1,\dots,\nu_r\}$ the set of Archimedean places of $F$ at which $B/F$ splits, and fix  isomorphisms 
\[
\iota_{\nu}\colon B\otimes_{F,\nu}\R\cong M_2(\R)\quad \text{for} \quad \nu\in\Sigma.
\]
Given an Eichler order $R$ in $B$, we denote by $\widetilde{\Gamma}$ the subgroup of $R^\times$ consisting of those elements with totally positive reduced norm. The group $\widetilde{\Gamma}$ maps in $\prod_{\nu\in\Sigma}(B\otimes_{F,\nu}\R)^\times\cong\mrm{GL}_2(\R)^{\Sigma}$ and hence it acts on $r$-copies of the Poincar\'e upper half plane $\cal{H}_{\Sigma}=\cal{H}_{\nu_1}\times\cdots\times\cal{H}_{\nu_r}$ via M\"obius transformations. Let $Z$ denote the center of $\mrm{GL}_2(\R)^{\Sigma}$. We suppose that $\Gamma:=\widetilde{\Gamma}/(\widetilde{\Gamma}\cap Z)$ is torsion-free so that the quotient $X_{B}:=\Gamma\backslash\cal{H}_{\Sigma}$ is a compact complex manifold. The holomorphic tangent bundle $\scr{T}$ of $X_{B}$ has a canonical decomposition into line bundles
\[
\scr{T}=\bigoplus_{\nu\in\Sigma}\scr{L}_\nu
\]
where $\scr{L}_\nu$ is the holomorphic line bundle associated to the $\nu$-th automorphy factor, and the image of first Chern class $c_1(\scr{L}_\nu)\in\mrm{H}^2(X_B,\Z)$ in $\mrm{H}^2_\mrm{dR}(X_B/\C)$ can be represented by 
\begin{equation}\label{ChernForm}
\frac{1}{4\pi i}\frac{\mrm{d}z_\nu\wedge\mrm{d}\overline{z}_\nu}{y_\nu^2}.
\end{equation}
Furthermore, for every $k\ge0$, the maximal torsion-free quotient of $\mrm{H}^k(X_B,\Z)$ carries a canonical $r$-plectic Hodge structure (Corollary \ref{plectic Hodge structure})
 which is preserved by Hecke operators (Section \ref{Functoriality}).
 Denote by $L_\nu\colon \mrm{H}^r(X_B,\Z)\to \mrm{H}^{r+2}(X_B,\Z)$ the morphism of $r$-plectic Hodge structures of bidegree $(1_\nu,1_\nu)$ given by cup product with the class $c_1(\scr{L}_\nu)$, and define the strongly primitive component 
\begin{equation}
\mrm{H}^r_\mrm{sp}(X_B,\Z)
\end{equation}
of  $\mrm{H}^r(X_B,\Z)$ as 
the maximal torsion-free quotient of the kernel of $\oplus_{\nu\in\Sigma}L_\nu$.

\begin{lemma}\label{carryPHS}
    The strongly primitive cohomology $\mrm{H}^r_\mrm{sp}(X_B,\Z)$ carries a canonical effective $r$-plectic Hodge structure of weight $\underline{1}\in\Z^r$.
\end{lemma}
\begin{proof}
Since the $L_\nu$'s are morphisms of plectic Hodge structures of bidegree $(1_\nu,1_\nu)$ the strongly primitive cohomology is endowed with an $r$-plectic Hodge structure. The claim follows after realizing that the explicit formula given in \eqref{ChernForm} allows us to compute
\[
\ker\big(L_\nu\colon \mrm{H}^r(X_B,\C)\to \mrm{H}^{r+2}(X_B,\C)\big)=\bigoplus_{\alpha_\nu=1\ \text{or}\ \beta_\nu=1}H^{\alpha,\beta}.
\]
\end{proof}

\begin{remark}
Matsushima and Shimura described the cohomology of $X_B$ in terms of automorphic forms in \cite{Matsu-Shimu}. An inspection of their result shows that the strongly primitive cohomology is the part of the cohomology of $X_B$ which is controlled by cuspidal automorphic representations.
\end{remark}

\noindent For every $\nu\in\Sigma$ let $\gamma_\nu\in R^\times$ be an element such that $\det(\iota_{\nu}(\gamma_\nu))<0$, and $\det(\iota_{\mu}(\gamma_\nu))>0$ if  $\mu\not=\nu$. Such an element exists because $F$ has narrow class number one. We define the non-holomorphic involution $\mrm{Fr}_{\nu_j}$ of $X_B$ by setting  
\begin{equation}
\mrm{Fr}_{\nu_j}(\tau_{\nu_1},\dots,\tau_{\nu_r})=\big(\iota_{\nu_1}(\gamma_{\nu_j})\tau_{\nu_1},\dots,\iota_{\nu_j}(\gamma_{\nu_j})\overline{\tau}_{\nu_j},\dots, \iota_{\nu_r}(\gamma_{\nu_j})\tau_{\nu_r}\big).
\end{equation}
    The involutions $\{\mrm{Fr}_{\nu}\}_{\nu\in\Sigma}$ acting on the cohomology of $X_B$ all commute with each other and with the Hecke operators (see for example \cite[Equation 5 $\&$ Section 3]{Greenberg}), but they are not morphisms of plectic Hodge structures. 
    Indeed, if we write $\mrm{Fr}_\beta=\prod_{\nu,\hspace{0.5mm} \beta_\nu=1}\mrm{Fr}_\nu$ for any $\beta\in\{0,1\}^r$, we find that 
    \begin{equation}\label{explicitFrob}
    \mrm{H}^r_\mrm{sp}(X_B,\C)=\bigoplus_{\alpha+\beta=\underline{1}}\mrm{Fr}_\beta^* \big(H^{\underline{1},\underline{0}}\big),
    \end{equation}
   that is, the base change to $\C$ of the strongly primitive cohomology is spanned by translates of holomorphic differential forms (\cite{Matsu-Shimu} $\&$ \cite[Theorem 1.3]{OdaGeneral}). Nevertheless, for any fixed $\nu\in \Sigma$, the operators $\{\mrm{Fr}_{\mu}\}_{\mu\not=\nu}$ are automorphism of the $\nu$-th Hodge structure of weight one
    \begin{equation}
    H_\Z(X_B,\nu):=\big(\mrm{H}^r_\mrm{sp}(X_B,\Z), F^{1_{\nu}}\big)
    \end{equation}
 attached to the strongly primitive cohomology as in \eqref{functorTOclassic} since 
\[
F^{1_{\nu}} \mrm{H}^r_\mrm{sp}(X_B,\C)=\bigoplus_{\beta,\hspace{0.5mm}\beta_\nu=0}\mrm{Fr}_\beta^* \big(H^{\underline{1},\underline{0}}\big).
\]

\begin{definition}
For every $\nu\in\Sigma$, we define the plectic Jacobian $\mrm{J}_\plectic(X_B,\nu)$ as the Jacobian of the Hodge structure $H_\Z(X_B,\nu)$.
\end{definition}

\begin{proposition}\label{AVs}
    For every $\nu\in\Sigma$, the plectic Jacobian $\mrm{J}_\plectic(X_B,\nu)$ is an abelian variety.
\end{proposition}
\begin{proof}
There is a decomposition of rational Hodge structures 
\[
H_\Q(X_B,\nu)=\bigoplus_{\chi}H_\Q(X_B,\nu)^\chi
\]
indexed by characters $\chi=\prod_{\mu\not=\nu}\chi_\mu\colon \prod_{\mu\not=\nu}\{\pm1\}\to\{\pm1\}$ such that $\mrm{Fr}_\mu$ acts on $H_\Q(X_B,\nu)^\chi$ as multiplication by $\chi_\mu(-1)$. Moreover, 
\[
\mrm{dim}_\bb{Q}\hspace{0.5mm} H_\Q(X_B,\nu)^\chi= 2\cdot \mrm{dim}_\C\hspace{0.5mm} H^{\underline{1},\underline{0}}.
\]
Let $\bb{T}_\bb{Q}^\mrm{good}$ denote the $\bb{Q}$-algebra generated by the good Hecke operators (for primes not dividing the level of the Eichler order $R$ and the discriminant of the quaternion algebra $B$) acting faithfully on $\mrm{H}^r_\mrm{sp}(X_B,\bb{Q})$ via endomorphisms of plectic Hodge structures (see Section \ref{Functoriality}). Using \eqref{explicitFrob} and the compatibility between the Hecke action and the action of the involutions $\{\mrm{Fr}_\nu\}_\nu$, we see that $\bb{T}_\bb{Q}^\mrm{good}$ is determined by its action on the space $H^{\underline{1},\underline{0}}$ of holomorphic differential forms. Therefore, the Jacquet-Langlands correspondence implies that $\bb{T}_\bb{Q}^\mrm{good}$ is isomorphic to a product $\prod_\xi L_\xi$ of totally real number fields and 
\[
H_\Q(X_B,\nu)^\chi_\xi:=H_\Q(X_B,\nu)^\chi\otimes_{\bb{T}_\bb{Q}^\mrm{good}}L_\xi
\]
is an effective rational Hodge structure of weight $1$ with real multiplication by $L_\xi$. The claim of the proposition then follows from Corollary \ref{ALGforHSwithRM}.
\end{proof}

\subsubsection{Plectic Oda conjecture.}
Let $E_{/F}$ be a modular elliptic curve corresponding to a quaternionic newform $f$ of some level $\Gamma$ for an indefinite quaternion algebra $B/F$. For every Archimedean place of $F$ $\nu\in\Sigma$ where $B/F$ is split, we consider the base change $E_{\nu}=E\times_{F,\nu}\C$, and we denote by $c_\nu$ the involution of $\mrm{H}^1(E_{\nu},\bb{Q})$ induced by action of complex conjugation on $E_{\nu}(\C)$. 
Let $\bb{T}_\bb{Q}$ denote the Hecke $\bb{Q}$-algebra generated by Hecke operators for primes not dividing the discriminant of $B/F$ acting faithfully on $\mrm{H}^r(X_B,\bb{Q})$ and denote by $\mrm{H}^r_\mrm{sp}(X_B,\bb{Q})_f$ the $f$-isotypic component of the strongly primitive cohomology. The following is a refinement of a classical conjecture of Oda (\cite{OdaGeneral}, Conjecture A).
\begin{conjecture}[{[Plectic Oda]}]\label{plecticODA}
    There is an isomorphism of rational $r$-plectic Hodge structures
    \[
    \mrm{H}^r_\mrm{sp}(X_B,\bb{Q})_f\cong \bigotimes_{\nu\in \Sigma}\mrm{H}^1(E_{\nu},\bb{Q})
    \]
    intertwining the action of $\mrm{Fr}_\nu$ with that of $c_\nu$ for every $\nu\in\Sigma$.
\end{conjecture}

\noindent We note that Conjecture \ref{plecticODA} together with the computation in Section \ref{ExTori} imply the existence of a morphism of abelian varieties 
\begin{equation}
\varphi_\nu\colon \mrm{J}_\plectic(X_B,\nu)\too E_\nu(\C)\otimes_\Z\bigotimes_{\mu\in\Sigma\setminus\{\nu\}}\mrm{H}^1(E_{\mu},\bb{Z})
\end{equation}
which should be thought of as a generalization of the  parametrization of elliptic curves by Jacobians of Shimura curves.
Moreover, it suggests the following conjecture which aims at elucidating the relation between the various plectic Jacobians attached to $X_B$. 
 \begin{conjecture}\label{plecticPARAM}
 There exists an abelian variety $\mrm{J}_{\plectic}(X_B)$ defined over $F$ and endowed with an $F$-rational Hecke action inducing a morphism  $\varphi\colon \mrm{J}_{\plectic}(X_B)\to E^{2^{\lvert\Sigma\rvert-1}}$ such that the analytification of $\varphi\otimes_{F,\nu}\C$ is canonically isomorphic to $\varphi_\nu$ for every $\nu\in\Sigma$.
 \end{conjecture}

\section{Plectic Abel--Jacobi maps}
Recall that $X_B=\Gamma\backslash\cal{H}_\Sigma$ is a compact quaternionic Shimura variety. The group $\Gamma$ is assumed to be torsion-free and it acts on $\cal{H}_\nu$ through its image $\nu(\Gamma)\le\mrm{PGL}_2(\R)$.
\begin{definition}
For any $\nu\in\Sigma$ we denote by $\cal{H}_\nu^\circ\subseteq \cal{H}_\nu$ the subset of those points with trivial stabilizer in $\Gamma$. We set $\cal{H}^{\circ}_\Sigma:=\prod_{\nu\in\Sigma}\cal{H}_\nu^\circ$ and $X^\circ_B:=\Gamma\backslash \cal{H}^{\circ}_\Sigma$, 
\end{definition}
\begin{remark}
Let $X^\mrm{CM}_B\subseteq X_B$ denote the subset of CM points, then $X^\mrm{CM}_B\subseteq X^\circ_B$.
\end{remark}

\noindent Let $M$ be a $\Gamma$-module. The subset $\cal{H}_\nu^\circ$ has been singled out because the higher homology groups of the tensor product $\mrm{Div}(\cal{H}_\nu^\circ)\otimes_\Z M$ with diagonal $\Gamma$-action vanish, i.e.,
\begin{equation}\label{keyvanishing}
\mrm{H}_k\big(\Gamma,\hspace{1mm} \mrm{Div}(\cal{H}_\nu^\circ)\otimes_\Z M \big)=0\qquad\forall\ k\ge1.
\end{equation}
Indeed, by definition there an isomorphism of $\Gamma$-modules
\[
\mrm{Div}(\cal{H}_\nu^\circ)\cong\bigoplus_{x\in\Gamma\backslash\cal{H}_\nu^\circ}\Z[\hspace{0.2mm}\Gamma\hspace{0.2mm}],
\]
and $\Gamma$-module $Z[\Gamma]\otimes_\Z M$ with diagonal $\Gamma$-action is isomorphic to the induced $\Gamma$-module $\mrm{Ind}^\Gamma_{\{1\}}(M)$.

\begin{definition}
Let $S\subseteq\Sigma$  be a subset with complement denoted by $S^c$. We define 
\[
\Z_S\big[\cal{H}^{\circ}_\Sigma\big]:=\bigotimes_{\nu\in S}\mrm{Div}(\cal{H}_\nu^\circ)\otimes\bigotimes_{\nu\in S^c}\mrm{Div}^0(\cal{H}_\nu^\circ).
\]
where $\mrm{Div}^0(\cal{H}_\nu^\circ)$ denotes the group of divisors of degree zero. 
\end{definition}
\begin{proposition}\label{kernel}
There is a short exact sequence
\[\xymatrix{
0\ar[r]&\mrm{H}_{\lvert\Sigma\rvert}\big(\Gamma,\Z\big)\ar[r]&\mrm{H}_0\big(\Gamma, \Z_\emptyset\big[\cal{H}^{\circ}_\Sigma\big]\big)\ar[r]& \Z\big[X^\circ_B\big].
}\]
\end{proposition}
\begin{proof}
Note that the free group $\Z\big[\cal{H}^{\circ}_\Sigma\big]$ equals $\Z_{\Sigma}\big[\cal{H}^{\circ}_\Sigma\big]$ and that $\Z\big[X^\circ_B\big]=\mrm{H}_0\big(\Gamma, \Z\big[\cal{H}^{\circ}_\Sigma\big]\big)$. 
Then, to prove the proposition it suffices to show that for any non-empty $S\subseteq\Sigma$ and any $\nu\in S$
\[
\ker\Big(\mrm{H}_0\big(\Gamma,\hspace{1mm} \Z_{S\setminus\{\nu\}}\big[\cal{H}^{\circ}_\Sigma\big]\big)\too \mrm{H}_0\big(\Gamma,\hspace{1mm} \Z_{S}\big[\cal{H}^{\circ}_\Sigma\big]\big)\Big)=\begin{cases}
\mrm{H}_r(\Gamma,\Z)&\text{if}\ S=\{\nu\},\\
0&\text{otherwise}.
\end{cases}
\]
To see that the claim holds, consider the short exact sequence of $\Gamma$-modules
\[\xymatrix{
0\ar[r]&\Z_{S\setminus\{\nu\}}\big[\cal{H}^{\circ}_\Sigma\big]\ar[r]& \Z_{S}\big[\cal{H}^{\circ}_\Sigma\big]\ar[rr]^-{\deg_\nu\otimes1}&& \bigotimes_{\mu\in S\setminus\{\nu\}}\mrm{Div}(\cal{H}_\mu^\circ)\otimes\bigotimes_{\mu\in S^c}\mrm{Div}^0(\cal{H}_\mu^\circ)\ar[r]&0.
}\]
If $\lvert S\rvert >1$, then we are done thanks to the observation \eqref{keyvanishing}. If $S=\{\nu\}$ we are left to prove that 
\begin{equation}\label{TOprove}
\mrm{H}_1\Big(\Gamma,\hspace{1mm}\bigotimes_{\mu\not=\nu}\mrm{Div}(\cal{H}_\mu^\circ)\Big)\cong \mrm{H}_r(\Gamma,\Z).
\end{equation}
For this, let $S\subseteq\Sigma$ be arbitrary, $\mu\not\in S$, and consider the short exact sequence of $\Gamma$-modules
\[\xymatrix{
0\ar[r]&\bigotimes_{\nu\in S\cup\{\mu\}}\mrm{Div}^0(\cal{H}_\nu^\circ)\ar[r]& \mrm{Div}(\cal{H}_\mu^\circ)\otimes\bigotimes_{\nu\in S}\mrm{Div}^0(\cal{H}_\nu^\circ)\ar[rr]^-{\deg_\mu\otimes1}&& \bigotimes_{\nu\in S}\mrm{Div}^0(\cal{H}_\nu^\circ)\ar[r] &0.
}\]
Once more, observation \eqref{keyvanishing} shows that taking homology we obtain the connecting isomorphisms
\begin{equation}\label{connISO}
\mrm{H}_{m+1}\Big(\Gamma,\hspace{1mm}\bigotimes_{\nu\in S}\mrm{Div}^0(\cal{H}_\nu^\circ)\Big)\cong \mrm{H}_{m}\Big(\Gamma,\bigotimes_{\nu\in S\cup\{\mu\}}\mrm{Div}^0(\cal{H}_\nu^\circ)\Big)\qquad\forall\ m\ge1.
\end{equation}
Thus, the isomorphism in \eqref{TOprove} follows by repeatedly applying \eqref{connISO}.
\end{proof}

\noindent  Using the cup product in de Rham cohomology we make the following identification
\[
\mrm{J}_\plectic(X_B,\nu)\cong \big(F^{1_\nu}\mrm{H}_\mrm{sp}^r(X_B,\C)\big)^\vee\big/\mrm{H}_r(X_B,\Z),
\]
and consider the homomorphism 
\begin{equation}\label{multiintegral}
\int^\Sigma\colon\mrm{H}_0\big(\Gamma, \Z_\emptyset\big[\cal{H}^{\circ}_\Sigma\big]\big)\too\mrm{J}_\plectic(X_B,\nu),\qquad \otimes_{j=1}^r(x_{\nu_j}-y_{\nu_j})\mapsto \left[\int_{y_{\nu_1}}^{x_{\nu_1}}\cdots \int_{y_{\nu_r}}^{x_{\nu_r}}(-)\right].
\end{equation}
It will be convenient for the proof of the next lemma to compute the singular homology of $X_B$ with the chain complex $(C_{\bfcdot}^\infty(X_B),\partial_{\bfcdot})$ of smooth cubical chains on $X_B$ (\cite{Massey}). An element $c\in C_n^\infty(X_B)$ is a non-degenerate smooth function $c:[0,1]^n\to X_B$ and the differential $\partial_n\colon C_n^\infty(X_B)\to C_{n-1}^\infty(X_B)$ is given by the formula
\begin{equation}\label{boundary}
\partial_n(c)=\sum_{k=1}^n(-1)^k[A_k(c)-B_k(c)] 
\end{equation}
where $A_k(c)(x_1,\dots,x_{n-1})=c(x_1,\dots,x_{k-1},0,x_{k+1},\dots,x_{n-1})$ can be thought of as the $k$-th ``front'' face of the cubical $n$-chain $c$,
and $B_k(c)(x_1,\dots,x_{n-1})=c(x_1,\dots,x_{k-1},1,x_{k+1},\dots,x_{n-1})$ as the $k$-th ``back'' face.

\begin{theorem}\label{compatiblehomology}
We have
\[
\mrm{H}_{r}(\Gamma,\Z)\subseteq \ker\int^\Sigma.
\]
\end{theorem}
\begin{proof}
 Integration over smooth cubical $n$-chains gives a morphism $\Upsilon\colon C_r^\infty(X_B)\to \mrm{H}_\mrm{dR}^r(X_B/\C)^\vee$. Moreover, Stokes' theorem implies that $\partial_{r+1}C_{r+1}^\infty(X_B)\subseteq\ker\Upsilon$, and de Rham's theorem gives the following exact sequence
 \[\xymatrix{
 0\ar[r]& \mrm{H}_r(X_B,\Z)/\mrm{tor}\ar[r]& C_r^\infty(X_B)\big/\ker\Upsilon\ar[r]^-{\partial_r}&C_{r-1}^\infty(X_B)\big/\partial_r\big(\ker\Upsilon\big).
 }
 \]
 The homomorphism $\int^\Sigma\colon\mrm{H}_0\big(\Gamma, \Z_\emptyset\big[\cal{H}^{\circ}_\Sigma\big]\big)\too\mrm{J}_\plectic(X_B,\nu)$ naturally factors through
\[
\xi\colon\mrm{H}_0\big(\Gamma,\Z_\emptyset\big[\cal{H}_\Sigma^{\circ}\big]\big)\too C_{r}^\infty(X_B)\big/\ker\Upsilon,\qquad \otimes_{j=1}^r(x_{\nu_j}-y_{\nu_j})\mapsto \int_{y_{\nu_1}}^{x_{\nu_1}}\cdots \int_{y_{\nu_r}}^{x_{\nu_r}}(-).
\]
Motivated by the description of the boundary map $\partial_r$ given in \eqref{boundary}, we define
\[
\zeta\colon\mrm{H}_0\big(\Gamma,\Z_\emptyset\big[\cal{H}^{\circ}_\Sigma\big]\big)\too \bigoplus_{j=1}^r \mrm{H}_0\big(\Gamma,\Z_{\{\nu_j\}}\big[\cal{H}^{\circ}_\Sigma\big]\big),\qquad \Delta\mapsto \big( (-1)^j\Delta\big)_{j=1}^r.
\]
Note that $\ker\zeta=\mrm{H}_r(\Gamma,\Z)$ by Proposition \ref{kernel}.
Then, the claim follows because there exists a morphism $\varpi\colon \mrm{Im}(\zeta)\to \partial_r \big(C_{r}^\infty(X_B)\big/\ker\Upsilon\big)$  making the following diagram (with exact rows) commute
\[\xymatrix{
0\ar[r]&\mrm{H}_r(\Gamma,\Z)\ar@{.>}[d]\ar[r]&\mrm{H}_0\big(\Gamma,\Z_\emptyset\big[\cal{H}^{\circ}_\Sigma\big]\big)\ar[r]^-\zeta\ar[d]^{\xi}& \mrm{Im}(\zeta)\ar@{.>}[d]^\varpi\\
0\ar[r]&\mrm{H}_r(X_B,\Z)/\mrm{tor}\ar[r]&C_{r}^\infty(X_B)\big/\ker\Upsilon\ar[r]^-{\partial_r}&\partial_r \big(C_{r}^\infty(X_B)\big/\ker\Upsilon\big).
}\]
\end{proof}

\begin{definition}
Let $\Z_\plectic[X^\circ_B]:=\mrm{Im}\big(\Z_\emptyset\big[\cal{H}^{\circ}_\Sigma\big]\to\Z[X^\circ_B]\big)$ denote the group of plectic zero-cycles supported on $X^\circ$. For any $\nu\in\Sigma$ the $\nu$-th plectic Abel--Jacobi map is the homomorphism
\[
\mrm{AJ}^\nu_\plectic\colon \Z_\plectic[X^\circ_B]\too \mrm{J}_\plectic(X_B,\nu)
\]
obtained from \eqref{multiintegral} using Proposition \ref{kernel} and Theorem \ref{compatiblehomology}.
\end{definition}

\bibliography{Plectic}

\begin{thebibliography}{FGM22}

\bibitem[Bea00]{Beauville}
A.~Beauville.
\newblock Complex manifolds with split tangent bundle.
\newblock In {\em Complex analysis and algebraic geometry}, pages 61--70. de
  Gruyter, Berlin, 2000.

\bibitem[BL99]{ComplexTori}
C.~Birkenhake and H.~Lange.
\newblock {\em Complex tori}, volume 177 of {\em Progress in Mathematics}.
\newblock Birkh\"{a}user Boston, Inc., Boston, MA, 1999.

\bibitem[Dar01]{IntegrationDarmon}
H.~Darmon.
\newblock Integration on $\cal{H}_p\times\cal{H}$ and arithmetic applications.
\newblock {\em Annals of Mathematics}, 154(3):589--639, 2001.

\bibitem[DL03]{Darmon-Logan}
H.~Darmon and A.~Logan.
\newblock Periods of {H}ilbert modular forms and rational points on elliptic
  curves.
\newblock {\em Int. Math. Res. Not.}, (40):2153--2180, 2003.

\bibitem[DR17]{DR2}
H.~Darmon and V.~Rotger.
\newblock Diagonal cycles and {E}uler systems {II}: {T}he {B}irch and
  {S}winnerton-{D}yer conjecture for {H}asse-{W}eil-{A}rtin {$L$}-functions.
\newblock {\em J. Amer. Math. Soc.}, 30(3):601--672, 2017.

\bibitem[Dru00]{Druel}
S.~Druel.
\newblock Vari\'{e}t\'{e}s alg\'{e}briques dont le fibr\'{e} tangent est
  totalement d\'{e}compos\'{e}.
\newblock {\em J. Reine Angew. Math.}, 522:161--171, 2000.

\bibitem[FG23a]{polyquadraticPlectic}
M.~Fornea and L.~Gehrmann.
\newblock On the algebraicity of polyquadratic plectic points.
\newblock {\em To appear in IMRN}, 2023.

\bibitem[FG23b]{plecticHeegner}
M.~Fornea and L.~Gehrmann.
\newblock Plectic {S}tark-{H}eegner points.
\newblock {\em Adv. Math.}, 414:108861, 42, 2023.

\bibitem[FGM22]{PlecticInvariants}
M.~Fornea, X.~Guitart, and M.~Masdeu.
\newblock Plectic {$p$}-adic invariants.
\newblock {\em Adv. Math.}, 406:108484, 26, 2022.

\bibitem[GH94]{GriffithsHarris}
P.~Griffiths and J.~Harris.
\newblock {\em Principles of algebraic geometry}.
\newblock Wiley Classics Library. John Wiley \& Sons, Inc., New York, 1994.
\newblock Reprint of the 1978 original.

\bibitem[Gor02]{GorenHMVs}
E.~Z. Goren.
\newblock {\em Lectures on {H}ilbert modular varieties and modular forms},
  volume~14 of {\em CRM Monograph Series}.
\newblock American Mathematical Society, Providence, RI, 2002.
\newblock With the assistance of Marc-Hubert Nicole.

\bibitem[GR04]{SGA1}
A.~Grothendieck and M.~Raynaud.
\newblock Rev\^etements \'etales et groupe fondamental ({SGA} 1), 2004.

\bibitem[Gre09]{Greenberg}
M.~Greenberg.
\newblock Stark-{H}eegner points and the cohomology of quaternionic {S}himura
  varieties.
\newblock {\em Duke Math. J.}, 2009.

\bibitem[Gri68]{Griffiths}
P.~A. Griffiths.
\newblock Periods of integrals on algebraic manifolds. {I}. {C}onstruction and
  properties of the modular varieties.
\newblock {\em Amer. J. Math.}, 90:568--626, 1968.

\bibitem[HM22]{hernMOLINA}
V.~Hernández and S.~Molina.
\newblock Plectic points and {H}ida-{R}ankin p-adic {L}-functions.
\newblock {\em Preprint}, 2022.

\bibitem[Lie68]{Lieberman}
D.~I. Lieberman.
\newblock Higher {P}icard varieties.
\newblock {\em Amer. J. Math.}, 90:1165--1199, 1968.

\bibitem[Mas80]{Massey}
W.~S. Massey.
\newblock {\em Singular homology theory}, volume~70 of {\em Graduate Texts in
  Mathematics}.
\newblock Springer-Verlag, New York-Berlin, 1980.

\bibitem[May]{MayNotes}
J.~P. May.
\newblock Notes on {D}edekind rings.
\newblock {\em http://www.math.uchicago.edu/~may/MISC/Dedekind.pdf}.

\bibitem[MS63]{Matsu-Shimu}
Y.~Matsushima and G.~Shimura.
\newblock On the cohomology groups attached to certain vector valued
  differential forms on the product of the upper half planes.
\newblock {\em Ann. of Math. (2)}, 78:417--449, 1963.

\bibitem[NS16]{PlecticNS}
J.~Nekov\'{a}\v{r} and A.~J. Scholl.
\newblock Introduction to plectic cohomology.
\newblock In {\em Advances in the theory of automorphic forms and their
  {$L$}-functions}, volume 664 of {\em Contemp. Math.}, pages 321--337. Amer.
  Math. Soc., Providence, RI, 2016.

\bibitem[NS19]{PlecticMHS}
J.~Nekov\'{a}\v{r} and A.~J. Scholl.
\newblock Plectic {H}odge {T}heory {I}.
\newblock {\em Preprint}, 2019.

\bibitem[Oda82]{Oda}
T.~Oda.
\newblock {\em Periods of {H}ilbert modular surfaces}, volume~19 of {\em
  Progress in Mathematics}.
\newblock Birkh\"{a}user, Boston, Mass., 1982.

\bibitem[Oda83]{OdaGeneral}
T.~Oda.
\newblock Hodge structures of {S}himura varieties attached to the unit groups
  of quaternion algebras.
\newblock In {\em Galois groups and their representations ({N}agoya, 1981)},
  volume~2 of {\em Adv. Stud. Pure Math.}, pages 15--36. North-Holland,
  Amsterdam, 1983.

\bibitem[PS08]{MHS}
C.~A.~M. Peters and J.~H.~M. Steenbrink.
\newblock {\em Mixed {H}odge structures}, volume~52 of {\em Ergebnisse der
  Mathematik und ihrer Grenzgebiete. 3. Folge. A Series of Modern Surveys in
  Mathematics [Results in Mathematics and Related Areas. 3rd Series. A Series
  of Modern Surveys in Mathematics]}.
\newblock Springer-Verlag, Berlin, 2008.

\bibitem[vdG88]{Geer}
G.~van~der Geer.
\newblock {\em Hilbert modular surfaces}, volume~16 of {\em Ergebnisse der
  Mathematik und ihrer Grenzgebiete (3) [Results in Mathematics and Related
  Areas (3)]}.
\newblock Springer-Verlag, Berlin, 1988.

\bibitem[Voi07]{Voisin}
C.~Voisin.
\newblock {\em Hodge theory and complex algebraic geometry. {I}}, volume~76 of
  {\em Cambridge Studies in Advanced Mathematics}.
\newblock Cambridge University Press, Cambridge, english edition, 2007.
\newblock Translated from the French by Leila Schneps.

\bibitem[Wei52]{Weil}
A.~Weil.
\newblock On {P}icard varieties.
\newblock {\em Amer. J. Math.}, 74:865--894, 1952.

\end{thebibliography}
\bibliographystyle{alpha}

\end{document}